\documentclass[10pt, oneside]{amsart}   	
\usepackage{geometry}                		
\usepackage{graphicx}				
\usepackage{amssymb}

\usepackage{amssymb}
\usepackage{amsthm}
\usepackage{graphicx}
\usepackage{mathtools}
\usepackage{amsmath,amsfonts}
\usepackage{epsfig}
\usepackage{comment}
\usepackage{color}
\usepackage{amsmath}

\newcommand{\Mod}{{\rm Mod}}

\newcommand{\R}{{\mathbb R}}

\newcommand{\N}{{\mathbb N}}

\newcommand{\spn}{N^{1,p}( X ) }

\newcommand{\ga}{\gamma}

\newcommand{\Ga}{\Gamma}

\newcommand{\cH}{\mathcal{H}}
\newcommand{\cF}{\mathcal{F}}

\newcommand{\defeq}{:=}
\newcommand{\Adm}{\rm Adm}

\DeclareMathOperator{\diam}{diam}

\DeclareMathOperator{\LIP}{LIP}

\def\vint_#1{\mathchoice
          {\mathop{\vrule width 6pt height 3 pt depth -2.5pt
                  \kern -8pt \intop}\nolimits_{#1}}%
          {\mathop{\vrule width 5pt height 3 pt depth -2.6pt
                  \kern -6pt \intop}\nolimits_{#1}}%
          {\mathop{\vrule width 5pt height 3 pt depth -2.6pt
                  \kern -6pt \intop}\nolimits_{#1}}%
          {\mathop{\vrule width 5pt height 3 pt depth -2.6pt
                  \kern -6pt \intop}\nolimits_{#1}}}

\theoremstyle{plain}
\newtheorem{theorem}{Theorem}[section]
\newtheorem{corollary}[theorem]{Corollary}
\newtheorem{lemma}[theorem]{Lemma}

\theoremstyle{definition}
\newtheorem{definition}[theorem]{Definition}

\newtheorem{remark}[theorem]{Remark}

\title{On The Sharp Lower Bound for Duality of Modulus}

\author{Sylvester Eriksson-Bique}
\address{Research Unit of Mathematical Sciences,
P.O.Box 3000,
FI-90014 Oulu, Finland}
\email{\tt sylvester.eriksson-bique@oulu.fi}

\author{Pietro Poggi-Corradini}
\address{Kansas State University,
Department of Mathematics,
138 Cardwell Hall,
Manhattan, KS 66506}
\email{pietro@math.ksu.edu}


\begin{document}
\maketitle

\begin{abstract} 
We establish a sharp reciprocity inequality for modulus in compact metric spaces $X$ with finite Hausdorff measure.  In particular, when $X$ is also homeomorphic to a planar rectangle, our result answers a question of K.~Rajala and M.~Romney. More specifically, we obtain a sharp inequality 
between the modulus of the family of curves connecting two disjoint continua $E$ and $F$ in $X$ and the modulus of the family of surfaces of finite Hausdorff measure that separate $E$ and $F$.
The paper also develops approximation techniques, which may be of independent interest.
\end{abstract}

\section{Introduction}
Modulus, which we define below in Section \ref{sec:preliminaries}, is a way of measuring the richness of a collection of curves, or more generally a collection of surfaces or even measures, and arises as the value of a convex minimization problem.
Ahlfors and Beurling showed in \cite{ahlfbehr} that for a topological rectangle $Q$ in the plane $\R^2$, if $\Gamma_1(Q)$ is the family of curves connecting a pair of opposite sides, and $\Gamma_2(Q)$ is the family of curves connecting the other pair of opposite sides, then their $2$-modulus values are related by the following reciprocal formula:
\begin{equation}\label{eq:explanation}
\Mod_2 \Gamma_1(Q)\cdot \Mod_2 \Gamma_2(Q)=1.
\end{equation}
 Later, see \cite{aikawohtsu,ziemercapacity}, this relation was generalized  to families of curves and separating surfaces in Euclidean spaces of higher dimension. More recently, see \cite{lohvansuu,joneslahti}, this was extended to more general metric spaces in slightly different forms, with some assumptions of the underlying space, such as doubling and the presence of a  Poincar\'e inequality. In the discrete setting, such inequalities were explored on graphs in \cite{pietroduality}.

Remarkably, Kai Rajala showed in \cite{rajalaunif}, that if the equality in Equation \eqref{eq:explanation} is replaced by a comparability of the form 
\begin{equation}\label{eq:quasi-reprocity}
\frac{1}{\kappa}\leq (\Mod_2 \Gamma_1(Q))^{1/2}(\Mod_2 \Gamma_2(Q))^{1/2} \leq \kappa,
\end{equation}
 for some $\kappa$ together with another technical assumption, then this provides a characterization of metric surfaces $X$ (with locally finite Hausdorff 2-measure) which are \emph{quasiconformal} to the plane. Rajala's result extends a long line of work intent on constructing quasiconformal uniformizations in the spirit of the classic uniformization of Riemann surfaces via conformal maps, see e.g. \cite{bonkkleiner, rajalaunif} for more discussion. Recently, it was observed by Rajala and Romney \cite{rajalaromney} that the lower bound in (\ref{eq:quasi-reprocity}) holds automatically whenever the underlying space is homeomorphic to $\R^2$ and has locally finite Hausdorff $2$-measure. They were able to establish this fact  with the specific constant $\kappa = 8000/\pi$ and conjectured in the same paper that the optimal constant should be $\kappa = 4/\pi$.
In particular, they give an example showing $\kappa$ cannot be smaller than $4/\pi$ (see the discussion on sharpness below). The purpose of our paper is to establish this conjecture and show that it holds also in a slightly more general context. 

To set the notation, suppose that $X$ is a  compact metric space with  finite $
\cH^N$-Hausdorff measure, for some $N\in\R$, with $N\ge 1$.  Suppose $E,F$ are two disjoint \emph{continua}, i.e., nonempty, compact connected sets, in  $X$. Let $\Gamma(E,F)$ be the family of curves connecting $E$ and $F$ in $X$. Also, let $\Sigma(E,F)$ be the family of topological boundaries $\partial U$ of open sets $U$ in $X$, such that $E \subset U$ and $F \subset {\rm int}(U^c)$. We think of $\partial U$ as a surface separating $E$ and $F$. Consider the corresponding family of measures $\Sigma_H(E,F)$ consisting of all Hausdorff measures of the form $\mathcal{H}^{N-1}|_{\partial U}$ that are finite. If we want to consider these notions relative to a subset $Q\subset X$, we write $\Gamma(E,F;Q),\Sigma_H(E,F;Q)$. Throughout the paper, $p \in (1,\infty)$ and $q$ is its dual exponent, namely $p^{-1}+q^{-1}=1$. For our next result, $X$ need not be homeomorphic to $\R^N$ or any of its subsets. However, later we will specialize to such a setting.

\begin{theorem}\label{thm:lowerbound}
Suppose that $X$ is a compact metric space, with finite $
\cH^N$-Hausdorff measure, for some real number $N\ge 1$. Set  $p,q \in (1,\infty)$ with $p^{-1}+q^{-1}=1$.
Let $E,F$ be two disjoint continua in $X$, with $\Ga(E,F)$ and $\Sigma_H(E,F)$ defined as above. Then, if $\Mod_p(\Gamma(E,F))>0$, the following inequality holds
\begin{equation}\label{eq:weakduality}
(\Mod_p \Gamma(E,F))^{\frac{1}{p}} (\Mod_q \Sigma_H(E,F))^{\frac{1}{q}} \geq \frac{v_N}{2v_{N-1}},
\end{equation}
where $v_k:=\frac{\pi^{k/2}}{\Ga\left(\frac{k}{2}+1\right)}$, for $k\ge 1$, and $\Ga\left(\frac{k}{2}+1\right)=\int_0^\infty x^{k/2}e^{-x}dx$ is the usual Gamma function.

Moreover, if $\Mod_p \Gamma(E,F) = 0$, then $\Mod_q \Sigma_H(E,F)=\infty$. 
\end{theorem}
For the definition of modulus we refer to Definition \eqref{eq:modulusdef}. In particular, in the setting of Theorem \ref{thm:lowerbound}, we always have $\Mod_p \Gamma(E,F)<\infty$, because the constant function $\rho(x) \equiv d(E,F)^{-1}$ is admissible. Note also that it is possible for $\Sigma_H$ to contains null-measures, in which case the modulus $\Mod_q \Sigma_H(E,F)$ becomes infinite.

As a corollary, when we restrict to planar metric spaces, we obtain the inequality conjectured by Rajala and Romney in \cite{rajalaromney}, which we extend also to the case $p\ne 2$. 
\begin{corollary}\label{cor:plane} Suppose that $Y$ is a metric space homeomorphic to $\R^2$, which has locally finite $\cH^2$-measure. Set  $p,q \in (1,\infty)$ and $p^{-1}+q^{-1}=1$. 
Assume $Q\subset Y$ is homeomorphic to $[0,1]^2$, and hence can be thought as a quadrilateral. Let the sides of $Q$ correspond to continua $A,B,C,D$ (in cyclic order). Set $\Gamma_1(Q):=\Gamma(A,C;Q)$,  the family of curves connecting $A$ and $C$ in $Q$, and $\Gamma_2(Q):=\Gamma(B,D;Q)$,  the family of curves connecting $B$ and $D$ in $Q$. Then, the following inequality holds: 
\begin{equation}\label{eq:weakduality-quadrilateral}
(\Mod_p \Gamma_1(Q))^{\frac{1}{p}} (\Mod_q\Gamma_2(Q))^{\frac{1}{q}}  \geq \frac{\pi}{4}.
\end{equation}
\end{corollary}
To prove Corollary \ref{cor:plane}, we will apply Theorem \ref{thm:lowerbound} with $Q=X$, since $X$ will be compact and have finite Hausdorff $\cH^2$-measure.  The corollary is sharp, as can be seen from the following example already mentioned in \cite[Example 2.2]{rajalaunif}. 
 
\noindent \textbf{Sharpness:} Take $Y=\R^2$ equipped with the $\ell^\infty$-distance $d((x_1,x_2),(y_1,y_2))=\max_{i=1,2}|x_i-y_i|$. Let $Q\defeq [0,1]^2$. We want to compute the modulus of the families of curves connecting the horizontal and vertical pairs of sides. 

Consider the Hausdorff measure $\cH^2$ with respect to the metric $d$. The usual scaling factor $v_N 2^{-N}$ in Equation \eqref{eq:hmeas-def}, is equal to $\pi/4$ when $N=2$. Since $\cH^2$ is a translation invariant and locally finite measure, there must be some constant $c$ so that $\cH^2(A)=c\lambda(A)$ for each Borel set $A$, where $\lambda$ is the usual Euclidean area measure. 
This constant can be determined by computing $\cH^2(Q)$. On one hand, we obtain $\cH^2(Q)\leq \pi/4$ by considering coverings of $Q$ by a grid of squares of side length $n^{-1}$ and sending $n\to\infty$. On the other hand, consider any countable cover $A_i$ of $Q$ with $Q \subset \bigcup_i A_i$. Each $A_i$ can be replaced by its bounding box $\tilde{A_i}$, which has the same diameter in the $\ell^\infty$-metric. Denote the diameter of a set $E$ with respect to $d$ by $\diam_d(E)$. Then, $\sum_i \diam_d(A_i)^2=\sum_i \diam_d(\tilde{A_i})^2\geq \sum_i \lambda(\tilde{A_i}) \geq 1$. Hence, accounting for the scaling factor in Equation \eqref{eq:hmeas-def}, together with the upper bound established before, we get that $\cH^2(Q)=\pi/4=c$. 
We remark, that the same could have also been established by the deep result of Kirchheim \cite[Lemma 6]{kircheim}.

Suppose now that $\rho$ is admissible for the curves connecting the left to the right hand side. By admissibility for horizontal curves, we get
$$\int_{0}^1 \!\!\int_0^1 \rho ~d\lambda \geq 1.$$
Thus, $\int_Q \rho ~d\cH^2\geq \frac{\pi}{4}$. An application of H\"older's inequality gives 
\[
\frac{\pi}{4}\le \left(\int_Q \rho^p d\cH^2\right)^{1/p}\left(\int_Q  d\cH^2\right)^{1/q}=\left(\int_Q \rho^p d\cH^2\right)^{1/p}\left(\frac{\pi}{4}\right)^{1/q}
\]
and thus 
\[
\int_Q \rho^p d\cH^2 \ge \frac{\pi}{4}\qquad\forall p\in (1,\infty).
\]
Minimizing over $\rho$ admissible, we get that $\Mod_p(\Gamma_1(Q)) \geq \frac{\pi}{4}$, for all $p\in(1,\infty)$. 
Conversely, the constant function $\rho\equiv 1$ is also admissible for $\Ga_1(Q)$, thus $\Mod_p(\Ga_1(Q))\le \int_Q d\cH^2=\frac{\pi}{4}$. Hence, $\Mod_p(\Gamma_1(Q)) = \frac{\pi}{4}$, for all $p\in(1,\infty)$, and,
by symmetry, $\Mod_q(\Gamma_2(Q))= \frac{\pi}{4}$ as well. Thus, in this example Equation (\ref{eq:weakduality-quadrilateral}) becomes
\[
(\Mod_p \Gamma_1(Q))^{\frac{1}{p}} (\Mod_q\Gamma_2(Q))^{\frac{1}{q}}  = \left(\frac{\pi}{4}\right)^{\frac{1}{p}}\left(\frac{\pi}{4}\right)^{\frac{1}{q}}=\frac{\pi}{4}.
\]
The proof of Theorem \ref{thm:lowerbound} rests on a co-area type estimate, analogously to \cite{rajalaromney},  and a Lipschitz approximation. However, to get the sharp constant we need a novel approximation scheme that yields sharper bounds. See Theorem \ref{thm:simplified} for a precise statement. This technique was first introduced in \cite{seb2020}, and, in the context of this paper, it yields the following result, which may be of independent interest. 
\begin{theorem}\label{thm:continuous-u} Let $p\in[1,\infty)$ with $q$ its dual exponent.
Let $X$ be a compact metric space with  finite $\cH^{N}$-measure, for some real number $N\ge 1$.
Let $E,F$ be two disjoint continua in $X$. Suppose that $u$ is a $\spn$-function such that $u=1$ on $F$ and $u=0$ on $E$. Then, there exists a sequence of Lipschitz  functions $u_i \in \spn$, so that $u_i=1$ on $F$ and $u_i=0$ on $E$,  with the property that for any Borel $\rho \in L^q(X)$ 
$$\limsup_{i 
\to 
\infty} \int_{0}^1 \int_{\partial\{u_i< t \}} \rho ~d\cH^{N-1} ~dt \leq \frac{2v_{N-1}}{v_N}\int \rho |\nabla u|_p ~d\cH^N .$$
\end{theorem}
The space $N^{1,p}(X)$  stands for the Newtonian space introduced in \cite{sha00}, which is an analogue of the classical Sobolev spaces. See also \cite{HKST07} for further background on such spaces. Below we will provide a short definition.

\noindent \textbf{Acknowledgements:} The first author was partially supported by the National Science Foundation under Grant No. DMS-1704215 and by the Finnish Academy under Research postdoctoral Grant No. 330048. The second author thanks the Department of Mathematics at UCLA, where this research started, for its generous support. The authors thank Kai Rajala and Mario Bonk for discussions on the topic and comments.

\section{Preliminaries}\label{sec:preliminaries}
In this paper, $(X,d)$ will denote a compact metric space with finite $\cH^N$-Hausdorff measure, for some real number $N\ge 1$ and $Y$ will denote a space homeomorphic to $\R^2$ with locally finite $\cH^2$-measure. The concepts of modulus, curve families and Newtonian (or Sobolev) spaces are defined in the same way for $X$ and $Y$. Throughout, we will use $N\ge 2$ to denote the Hausdorff dimension of the space.  The spaces will be equipped with the Hausdorff measure $\cH^N$, and it is with respect to these measures that we define the Lebesgue spaces  $L^p(X)$, for $p \in [1,\infty]$, and the notion of almost everywhere. Where needed, the norm on $L^p(X)$ will be denoted by $\|\cdot \|_{L^p}$. For purposes of normalization, we note that the Hausdorff measure is defined as $\cH^N(A) \defeq \lim_{\delta \to 0} \cH^N_\delta(A)$, where 
\begin{equation}\label{eq:hmeas-def}
\cH^N_\delta(A) \defeq \frac{v_N}{2^N}\inf \left\{ \sum_{i} \diam(A_i)^N : A \subset \bigcup_i A_i, \ \diam(A_i) < \delta \right\} 
\end{equation}
where $v_k:=\frac{\pi^{k/2}}{\Ga(\frac{k}{2}+1)}$, and $\Ga(t)$ is the usual Gamma function. Also, $\diam(E)=\sup_{x,y\in E} d(x,y)$ is the diameter of a set $E$.

Curves in $X$ are continuous maps $\gamma: I\to X$, defined on some compact non-empty interval $I\subset \R$. Let $\Gamma(X)$ be the family of all such curves. We are interested in the families $\Gamma(E,F)$ (or $\Gamma(E,F;Q)$) of all the curves which connect $E$ to $F$ (resp. in $Q$ for some  $Q \subset X$). Throughout, $E$ and $F$ will be two disjoint compact connected non-empty subsets of $X$.

Also let $M(X)$ denote the family of all finite Radon measures on $X$.  In particular, $\Sigma_H(E,F)$ will consist of all finite measures $\cH^{N-1}|_{\partial U}$, given an open set $U$ with $E \subset U$ and $F \subset {\rm int}(U^c)$. Here,  ${\rm int}(A)$ denotes the interior of the set $A\subset X$.

Let $\Gamma \subset \Gamma(X)$ be any family of curves. We say a non-negative Borel function $\rho: X \to [0,\infty]$ admissible for $\Ga$, and write $\rho\in\Adm(\Ga)$, if $\int_\gamma \rho \, d s \geq 1$, for each rectifiable $\gamma \in \Gamma(X)$. Here $ds$ represents the arc-length parametrization of $\ga$, see \cite[Chapter 5]{HKST07} for more details. Then, we define the $p$-modulus (for $1\leq p < \infty$) of a curve family as
\begin{equation}\label{eq:modulusdef}
\Mod_p(\Gamma) \defeq \inf_{\rho\in\Adm(\Ga)} \int_X \rho^p \, d\cH^N,
\end{equation}
When $p=\infty$, we take the infimum of $\|\rho\|_{L^\infty}$. If $\Sigma \subset M(X)$, then we define $\Mod_p(\Sigma)$ by replacing the admissibility condition with $\int \rho \, d\sigma \geq 1$ for each measure $\sigma \in \Sigma$. 

We say that a property holds for $p$-almost every curve, if the set of curves $\Gamma_0$ for which the property does not hold has $\Mod_p(\Gamma_0)=0$.

If $f: X\to[-\infty,\infty]$ is measurable, we call $g: X\to[0,\infty]$ an upper gradient for $f$, if for every rectifiable curve $\gamma:[0,1]\to X$ we have
\begin{equation}\label{eq:ug}
\int_\gamma g ~ds \geq |f(\gamma(0))-f(\gamma(1))|,
\end{equation}
where we interpret $|\infty-\infty|=\infty$. 
 We say that $f\in \spn$ if $f\in L^p(X)$ and $f$ has an upper gradient $g\in L^p(X)$. There is a function $|\nabla f|_p$, called a minimal $p$-weak upper gradient for which $|\nabla f|_p\leq g$ ($\cH^N$-a.e.) for each upper gradient $g$, and for which estimate \eqref{eq:ug} holds for $p$-a.e. curve. Although this notation suggests that a point-wise ``gradient'' $\nabla f$ may exist, this is not necessarily the case. However, we wish to connect this notation to the Euclidean notion, where the expression is actually the norm of the distributional gradient.

\begin{remark}\label{rmk:vitalicarath} When $g=|\nabla f|_p$, for any $\epsilon>0$ there is a Borel function $h: X\to[0,\infty]$ so that $\int h^p \, d\cH^N\leq \epsilon/2$, and $\int_\gamma h \, ds = \infty$ for every curve $\gamma$ Inequality (\ref{eq:ug}) does not hold. Then, $\tilde{g}=\max(|\nabla f|_p,h)$ is an upper gradient with $\int \tilde{g}^p \, d\cH^N \leq \int |\nabla f|_p^p \, d\cH^N+ \epsilon/2$. By Vitali-Carath\'eodory, we can choose a lower semicontinous $g_\epsilon$, so that $\tilde{g}\leq g_\epsilon$ and $\int g_\epsilon^p \,d\cH^N \leq \int \tilde{g}^p \, d\cH^N+ \epsilon/2 \leq \int |\nabla f|_p^p\, d\cH^N+\epsilon$. With the same argument, the infimum in Definition \eqref{eq:modulusdef} can be taken over lower semicontinuous functions. We refer the reader to \cite[Chapters 4 and 5]{HKST07} for more details on modulus and a proof of Vitali-Carath\'eodory
\end{remark}

We also need the following simple case of modulus equalling capacity (see e.g. \cite[Section 2.11]{heinonenkoskela}). We don't wish to introduce capacity without really needing it here, so we prefer to give only a weak statement whose proof can be given directly.

\begin{lemma}\label{lem:modcap} Suppose that $g$ is admissible for $\Mod_p(\Gamma(E,F))$, then there is a function $u \in \spn$ so that $u|_E = 0$ and $u|_F=1$ with $|\nabla u|_p \leq g$ almost everywhere.
\end{lemma}

\begin{proof} Define $u(x)= \min\left(\inf_{\gamma} \int g \, ds,1\right)$, where the infimum is over paths $\gamma$ connecting $E$ to $x$. By \cite[Corollary 1.10]{jarvempaa} the function $u(x)$ measurable. Note that  $u|_E = 0$ by definition. Also, $u|_F=1$, because $g$ is admissible. 
Since $X$ is compact, we have $u(x) \in L^p(X)$. Next, we show that $g$ is an upper gradient for $u$ and this will imply that $u \in \spn$ and $|\nabla u|_p \leq g$. 

Let $\gamma$ be any curve joining $x=\gamma(0)$ and $y=\gamma(1)$. We will show that $u(y)-u(x) \leq \int_\gamma g \, ds$. Then, reversing the curve gives the desired bound \eqref{eq:ug}. If $u(x) = 1$, the inequality is immediate since then $u(y)-u(x) \leq 0$. Therefore, we can assume that $u(x) = \inf_{\gamma_x} \int g \, ds$ with the infimum is taken over all the curves $\gamma_x$  joining $E$ to $x$. Fix one such curve $\gamma_x$. Define a curve $\gamma_y$ by concatenating $\gamma_x$ with $\gamma$, and parametrize it so that $\gamma_y(0)=\gamma_x(0) \in E$. Then,
$$u(y) \leq \int_{\gamma_y} g \, ds = \int_{\gamma_x} g \, ds + \int_\gamma g \, ds.$$
Taking an infimum over $\gamma_x$ yields the desired bound.
\end{proof}
In order to get Corollary \ref{cor:plane} from Theorem \ref{thm:lowerbound}, we need a general modulus statement.
\begin{lemma}\label{lem:modulus-contained} Suppose that $\Sigma_1,\Sigma_2 \subset M(X)$ are two families of measures, and for each $\sigma_2\in \Sigma_2$ we have some measure $\sigma_1\in \Sigma_1$ with $\sigma_1\leq\sigma_2$, then $\Mod_p(\Sigma_2)\leq\Mod_p(\Sigma_1)$.
\end{lemma}
\begin{proof}
The claim follows because any Borel function admissible for $\Sigma_1$ will automatically be admissible for $\Sigma_2$.
\end{proof}
We will also need the following useful topological fact. Recall that, if $Y$ is a metric space homeomorphic to $\R^2$, which has locally finite Hausdorff $\cH^2$-measure, then a quadrilateral $Q\subset Y$ is a subset homeomorphic to $[0,1]^2$. The homeomorphic images of the sides will be denoted $A,B,C,D$, in cyclic order. Note that the notion of ``opposite'' edges does not depend on the orientation of the boundary. In the following, the relative boundary of a set $S \subset Q$ is denoted as $\partial_Q S$
\begin{lemma}\label{lem:contained}
Suppose that $Q\subset Y$ is a quadrilateral as defined above. Let $U$ be an open set in $Y$ such that $A\subset U$ and $C\subset {\rm int}(U^c)$, with $\cH^1\mid_{\partial_Q U}$  finite so that it is a measure in $\Sigma_H(A,C;Q)$. 

Then, there is a simple rectifiable curve $\gamma$ connecting $B$ to $D$, with $\gamma \subset \partial_Q U$ and $\cH^{1}|_{\partial_Q U}\geq \cH^{1}|_{\gamma}$.
\end{lemma}

This claim is classical, but we indicate a proof for the sake of completeness.

\begin{proof} By Zorn's lemma, there is a minimal compact set $K \subset \partial_Q U$ which still separates $A$ from $C$ in that it does not contain any strictly smaller separating compact set. Indeed, it is enough to show that  for every chain $\{K_j\}_1^\infty$ of compact sets separating $A$ from $C$ with $K_{j+1}\subset K_j$, the intersection $K_\infty:=\cap_j K_j$ is also compact and must still separate $A$ from $C$. Assume that $K_\infty$ does not separate. Then, there is a curve $\ga$ from $A$ to $C$ in the complement of $K_\infty$, i.e., with $d(\ga,K_\infty)>0$. However, there are points $w_j\in\ga\cap K_j$ for every $j$ and by compactness a subsequence will converge to a point $w_\infty$. Note that $w_j\in K_i$ for all $j\ge i$. So $w_\infty\in K_i$, for every $i\ge 1$. Hence, $w_\infty\in K_\infty$, which leads to a contradiction.

Moreover, this minimal separating compact set $K$ must be connected.
If  not, then $K$ could be expressed as a union of two disjoint non-empty compact subsets $K_1,K_2$. Since $K_1 \cap K_2 = \emptyset$, by Janiszewiski's theorem, see e.g. \cite[p.110]{new64}, either $K_1$ or $K_2$ must separate. However, this is a contradiction to minimality of $K$. Therefore, $K$ must be a continuum. 

Since $K$ is connected and has finite Hausdorff measure, then by (the argument in) \cite[Lemma 3.7]{Schul} $K$ is rectifiable. Thus, we can find a rectifiable curve $\gamma:I \to \partial_Q U$ which separates $A$ from $D$. The curve $\gamma$ must intersect $B$ and $D$, and thus must contain a sub-curve $\gamma|_J$ for some $J \subset I$ which connects $B$ to $D$. By possibly removing loops, i.e. choosing the shortest curve contained in the image of $\gamma$ and connecting $B$ to $D$, we can insist that $\gamma$ be simple. The inclusion $\gamma  \subset \partial_Q U$ gives $\cH^{1}|_{\partial U}\geq \cH^{1}|_{\gamma}$.
\end{proof}

We will also need the following version of Arzel\`a-Ascoli's theorem. 

\begin{lemma}\label{lem:arzelaascoli}Assume that $Z$ is a complete metric space and that $L\in (0,\infty)$. Suppose that $\gamma_n:[0,1]\to Z$ is a sequence of $L$-Lipschitz curves with $A_t\defeq \{\gamma_n(t) : n \in \N\}$ precompact for each $t\in [0,1]$. Then, there exists a subsequence $\gamma_{n_k}$ which converges uniformly to a curve $\gamma:[0,1]\to Z$.
\end{lemma}

The proof of this version is completely classical (see for instance \cite[Theorem 4.25]{walterreal}). 
However, we remark, that instead of assuming that $Y$ is compact, we assume that the set of pointwise values $A_t$ is pre-compact. Indeed, the usual proof of first constructing $\gamma(t)$ on a dense subset of rational values $t$ together with a diagonal argument only requires the pointwise values to be precompact. The precompactness of the sets $A_t$, in our application, will be shown to follow from the fact that $A_t$ is close to a compact subset, except for finitely many points. This argument allows us to perform a limit process in the ambient space $Z \defeq \ell_\infty(\N)$, in which $X$ can be embedded using the Kuratowski embedding. Thus, Arzel\`a-Ascoli argument can still be applied, despite the lack of compactness of $Z$.

\section{Lipschitz approximation}

A function $f:X\to Y$ between two metric spaces is Lipschitz if  $\sup_{x,y\in X, x\neq y}\frac{d(f(x),f(y))}{d(x,y)}$ is finite, where $d$ denotes the distance both on $X$ and $Y$.
Given a Lipschitz function $f:X \to Y$ and a subset $A \subset X$ define the Lipschitz constant of $F$ on $A$ as
$$\LIP[f](A)\defeq \sup_{x,y \in A, x \neq y} \frac{d(f(x),f(y))}{d(x,y)}.$$

Recall Eilenberg's inequality:  for any Borel set $A\subset X$ and any Lipschitz function $u:X \to \R$:
\begin{equation}\label{eq:eilenberg}
\int_{-\infty}^\infty \cH^{N-1}(u^{-1}(t)\cap A) ~dt \leq \frac{2v_{N-1}}{v_N}\LIP[u](A) \cH^N(A).
\end{equation}

The inequality is due to \cite[Theorem 1]{eilenbergoriginal}, where a relatively simple proof is presented using an upper integral. The reader can consult \cite{eilenbergeasy} for a discussion and further references. The version we use combines \cite[Theorem 1]{eilenbergoriginal} with the following remark and lemma. 
\begin{remark} There is a measurability consideration that is not explained in the papers cited above, which can be bypassed by expressing the integral on the left hand-side of Equation \eqref{eq:eilenberg} as an upper Lebesgue integral. Namely, the fact that the map $t\to \cH^{N-1}(u^{-1}(t)\cap A)$ is measurable, if $A\subset X$ is a Borel set, when $\cH^N(X)<\infty$ and $u:X\to\R$ is Lipschitz. To be self-contained, we present here an argument for this, that suffices for our purposes. As remarked in \cite[Remark 1.2]{eilenbergeasy}, the result holds in greater generality. The interested reader may also consult the beautiful treatise by Dellacherie related to this point \cite[Chapitre VI]{dellacherie}.
\end{remark}
\begin{lemma} Suppose that $\cH^{N}(X)<\infty$ and $X$ is compact. If $u: X \to \R$ is Lipschitz, and $A\subset X$ is Borel, then $t\to\cH^{N-1}(A \cap u^{-1}(t))$ is measurable.
\end{lemma}
\begin{proof}
If $A$ is a compact set, then the map $t\to \cH^{N-1}_\delta(u^{-1}(t)\cap A)$ will be upper semi-continuous. This can be seen as follows. Fix $t\in \R$, and $a\in (0,\infty)$ so that $\cH^{N-1}_\delta(u^{-1}(t) \cap A) < a$. By compactness, we can choose a finite open cover $\{A_i\}_{i=1}^M$ of $A$ with $\diam(A_i) < \delta$  with $\sum_{i=1}^M \frac{v_N}{2^N}\diam(A_i)^N < a$. Here, we use the fact that in Equation (\ref{eq:hmeas-def}), we chose to define $\cH^{N-1}_\delta$ using a strict inequality on diameter. Since $X$ is compact and $u$ is Lipschitz, for all $t'$ close enough to $t$, we will also have $u^{-1}(t')\cap A \subset \bigcup_{i=1}^M A_i$. From this, the upper semi-continuity follows. Sending $\delta\to 0$ gives that $t\to\cH^{N-1}( u^{-1}(t)\cap A)$ is Borel, when $A$ is compact. 

When $A$ is simply Borel, notice that $A$ has finite $\cH^N$-measure, since $\cH^N(X)<\infty$. Hence, we can exhaust $A$ by an increasing sequence of compact subsets $K_i$ so that $\cH^{N}(A\setminus \bigcup_i K_i) = 0$. The Eilenberg inequality involving an upper integral, \cite[Theorem 1]{eilenbergoriginal}, shows that for a.e. $t\in \R$ we have $\cH^{N-1}(u^{-1}(t)\cap (A \setminus \bigcup_i K_i))=0$ and that $\cH^{N-1}(u^{-1}(t)\cap A)$ is finite, and thus almost everywhere $\cH^{N-1}(u^{-1}(t)\cap A) = \lim_{i\to\infty} \cH^{N-1}(u^{-1}(t)\cap K_i)$. Measurability follows from this limit.
\end{proof}
\vskip.3cm

\begin{definition}\label{def:loc-lip-ug}
 We call a non-negative function $g:X \to [0,\infty)$ a local Lipschitz upper gradient for $f:X \to \R$, if for every $x \in X$, there exists a $r_x>0$, so that for each $r\in(0,r_x)$
\begin{equation}\label{eq:lip-ug}
\LIP[f](B(x,r)) \leq \sup_{y \in B(x,2r)}g(y).
\end{equation}
\end{definition}

\begin{lemma}\label{lem:local-lipshitz} Suppose that $X$ is compact and that $g$ is a continuous local Lipschitz  upper gradient for a non-negative Lipschitz function $u: X \to \R$. For every Borel set $A$ we have
\begin{equation}\label{eq:eil-lip-ug}
\int_{-\infty}^\infty \cH^{N-1}(u^{-1}(t)\cap A) ~dt \leq \frac{2v_{N-1}}{v_N}\int_A g ~d\cH^{N}.
\end{equation}

Moreover, for any Borel function $\rho:X \rightarrow[0,\infty]$, we have
\begin{equation}\label{eq:eil-lip-ug-moreover}
\int_{-\infty}^\infty \int_{u^{-1}(t)} \rho ~d\cH^{N-1} ~dt \leq \frac{2v_{N-1}}{v_N}\int \rho g ~d\cH^{N}.
\end{equation}
\end{lemma}

\begin{proof}
Approximating $\rho$ by simple functions shows that Inequality (\ref{eq:eil-lip-ug-moreover}) follows from Inequality (\ref{eq:eil-lip-ug}). Therefore, we will show that Inequality (\ref{eq:eil-lip-ug}) holds.

It is enough to consider the case $\cH^{N}(A) > 0$, because when $\cH^N(A)=0$, Eilenberg's Inequality (\ref{eq:eilenberg}) implies that the left hand-side of Inequality (\ref{eq:eil-lip-ug})  vanishes. Since $X$ is compact, $g$ is uniformly continuous. Therefore, if $\epsilon>0$ is arbitrary, then we can find finitely many balls $\{B(x_i,r_i)\}_{i=1}^n$ that cover $X$, and thus $A$, so that and $r_i \leq r_{x_i}$, and with the property that 
\begin{equation}\label{eq:unif-cont-g}
 y \in B(x_i,r_i) \quad\Longrightarrow\quad |g(y)-g(x_i)| < \frac{\epsilon v_N}{4v_{N-1}\cH^N(A)}. 
\end{equation}
In particular, by Inequality (\ref{eq:lip-ug}),  we have,  for $i=1,\dots,n,$
$$\LIP[u](B(x_i,r_i)) \leq g(x_i)+\frac{v_N\epsilon}{4v_{N-1}\cH^N(A)}.$$
Define, inductively $K_1 := A\cap B(x_1, r_1)$, and $K_i := A \cap B(x_i,r_i) \setminus \bigcup_{j=1}^{i-1} B(x_j,r_j).$ On each $K_i$ apply Eilenberg's Inequality \eqref{eq:eilenberg}, to get
$$\int_{-\infty}^\infty \cH^{N-1}(u^{-1}(t)\cap K_i) ~dt \leq \frac{2v_{N-1}}{v_N}\cH^{N}(K_i) g(x_i) +\frac{\epsilon \cH^{N}(K_i)}{2\cH^N(A)} .$$
By Property (\ref{eq:unif-cont-g}),
$$\frac{2v_{N-1}}{v_N}\cH^{N}(K_i) g(x_i) \leq \frac{2v_{N-1}}{v_N}\int_{K_i} g(x) ~d\cH^{N} + \epsilon \frac{\cH^{N}(K_i)}{2\cH^N (A)}.$$
Summing these estimates over $i=1, \dots, n$ gives that
$$\int_{-\infty}^\infty \cH^{N-1}(u^{-1}(t)\cap K_i) ~ds \leq \frac{2v_{N-1}}{v_N}\int_A g(x) ~d\cH^{N} + \epsilon,$$
for any $\epsilon>0$, and then sending $\epsilon \to 0$ completes the proof.
\end{proof}
The next result is the main Lipschitz approximation scheme needed for our purposes.
\begin{theorem}\label{thm:simplified} Let $p\in [1,\infty)$. Assume $(X,d)$ is a compact metric space. Let $E,F \subset X$ be two disjoint, non-empty, compact, connected sets. Suppose that $u: X \to \R$ is in $\spn$ with $u|_E=0, u|_F=1$. Assume further that $u$ has a lower semi-continuous non-negative upper gradient $g \in L^p$ and that there is an $\epsilon > 0$, so that $g \geq \epsilon$ on $X$. 

Then, there is a sequence of Lipschitz functions $u_i$, with $0 \leq u_i \leq 1$, $u_i|_E=0, u_i|_F=1$, so that each has a continuous local Lipschitz upper gradients $g_i$, as in Definition \ref{def:loc-lip-ug}, that converge to $g$ in $L^p(X)$.
\end{theorem}
\begin{proof}
Let $\tilde{g}_i \nearrow g$ be a sequence of continuous functions converging to $g$ pointwise so that $\tilde{g}_i \geq \epsilon$. Then $\tilde{g}_i \to_{L^p} g$ by dominated convergence. By iteratively redefining $\tilde{g}^*_j \defeq \max_{i=1, \dots, j}\{\tilde{g}_{i}\},$ and simplifying notation, we can insist $\tilde{g}_i \leq \tilde{g}_j$ for $i \leq j$.


The goal is to build the approximating function from its gradient. To that end,  define
\begin{equation}\label{eq:grad-integrate}
\cF_{i}(x) \defeq \inf_{p_0, \dots, p_n} \sum_{k=0}^{n-1}\tilde{g}_i(p_k) d(p_k,p_{k+1}),
\end{equation}
where the infimum is taken over all chains $\{p_0, \dots, p_n\}$ of points in $X$, such that $p_0\in E,p_n=x$ and $d(p_k,p_{k+1}) \leq \frac{1}{i}$ for $k=0,\dots, n-1$. Such chains are called $(x,i)$-{\it admissible}. We force the upperbound of $1$ by setting $\tilde{u}_i(x) := \min(\cF_i(x),1)$.\\

\noindent \textbf{Claim 1:} {\it $\tilde{u}_i$ has $\tilde{g}_i$ as a local Lipschitz upper gradient. }


Let $x,y\in X$. Since the map $z \to \min\{z,1\}$ is a contraction, we have $$|\tilde{u}_i(x)-\tilde{u}_i(y)|\leq |\cF_i(x)-\cF_i(y)|.$$ Assume first that $d(x,y) \leq \frac{1}{i}$. If $\cF_i(x)<\infty$,  then by definition (\ref{eq:grad-integrate}),
$$\cF_i(y) \leq \cF_i(x) + \tilde{g}_i(x)d(x,y).$$ Indeed,  any $(x,i)$-admissible chain $E\ni p_0,\dots,p_n=x$ can be extended to a $(y,i)$-admissible chain by adding $p_{n+1}=y$. Since $\tilde{g}_i$  is continuous, this implies that $\cF_i(y)$ is finite as well. By symmetry, either $\cF_i(x)$ and $\cF_i(y)$ are both finite, in which case \[|\tilde{u}_i(x)-\tilde{u}_i(y)|\leq |\cF_i(x)-\cF_i(y)| \leq \max(\tilde{g}_i(x),\tilde{g}_i(y)) d(x,y);\] Or $\cF_i(x)=\cF_i(y)=\infty$, and then $\tilde{u}_i(x)=\tilde{u}_i(y)=1$. 
In conclusion, we get $|\tilde{u}_i(x)-\tilde{u}_i(y)| \leq \max(\tilde{g}_i(x),\tilde{g}_i(y)) d(x,y)$, whenever $d(x,y) \leq \frac{1}{i}$. Choosing $r_x=\frac{1}{2i}$ for each $x\in X$ gives that $\tilde{g}_i$ is a local Lipschitz upper gradient for $\tilde{u}_i$.

Now assume that $d(x,y) \geq \frac{1}{i}$. Since $0\le \tilde{u}_i\le 1$, we have
$$|\tilde{u}_i(x)-\tilde{u}_i(y)|\leq 1\leq i d(x,y),$$ which shows that $\tilde{u}_i$ is Lipschitz. Therefore, we have shown that $\tilde{u}_i$ is Lipschitz with $\tilde{g}_i$ as local Lipschitz upper gradient. \\

\noindent \textbf{Claim 2:} Let  $a_i \defeq \inf\limits_{x \in F} \tilde{u}_i(x).$ We claim that $\lim_{i\to \infty} a_i = 1$. \\

By definition of $\tilde{u}_i$, we have $a_i\leq 1$ for each $i$. Suppose the claim does not hold. By passing to a subsequence  we may assume that $\lim_{i \to \infty}a_i < 1-\delta$ for some $\delta>0$. We will assume that  $X$ is embedded isometrically in $\ell_\infty(\N)$ and we identify it with its image. Such an embedding can be found using a Kuratowski embedding, since $X$ is separable, see for instance \cite[p. 99]{hei01}. For each $i\geq 1$, extend $\tilde{g}_i$ to be a continuous function on $\ell_\infty(\N)$ using the Tietze Extension theorem,  see  \cite{milman97}. 
Redefining $\tilde{g}^*_j := \max\{\max(\tilde{g}_{i},\epsilon) : i = 1, \dots, j\},$ and simplifying notation, we can insure that  $\epsilon \leq \tilde{g}_i \leq \tilde{g}_j$ for $i \leq j$, on the full ambient space $\ell_\infty(\N)$.

Since $\lim_{i\to\infty} a_i <1-\delta$, we can find discrete $(x,i)$-admissible chains $p_0^i, \dots p_{n(i)}^i$, for each $i \in \N$, so that $p_0^i \in E$,  $p_{n(i)}^i \in F$, and with $p_k^i \in X$ and $d(p^i_k,p^i_{k+1}) \leq \frac{1}{i}$ for $k=0,\dots,n(i)$, so that
\begin{equation}\label{eq:discchain}
\tilde{u}_i(p_{n(i)}^i)=\cF_{i}(p_{n(i)}^i) \leq \sum_{k=0}^{n(i)-1}\tilde{g}_i(p^i_k)  d(p^i_k,p^i_{k+1})\leq 1-\delta/2.
\end{equation}
By construction, we also have $\epsilon d(p^i_k,p^i_{k+1}) \leq \tilde{g}_i(p^i_k)  d(p^i_k,p^i_{k+1})$ for each $i,k$. Therefore, Equation \eqref{eq:discchain} gives
\begin{equation}\label{eq:li}
L_i\defeq \sum_{k=0}^{n(i)-1} d(p^i_k,p^i_{k+1})\leq \frac{1}{\epsilon}.
\end{equation}
Set $t^i_0\defeq 0$ and $t^i_k \defeq \sum_{l=0}^{k-1} d(p^i_l,p^i_{l+1})/L_i$ for $k=1,\dots, n(i)$. Define $\gamma_i:[0,1]\to \ell_\infty(\N)$ as $\gamma_i(t^i_k) \defeq p^i_k$ for $k=0,\dots, n(i)$ and extend $\gamma_i(t)$ to the interval $ [t^i_k,t^i_{k+1}]$ by linear interpolation in $\ell_\infty(\N)$. If $k<l$ then, by the triangle inequality, 
\[
d(\gamma_i(t_k^i),\gamma_i(t_l^i)) = d(p_k^i,p_l^i) \leq\sum_{s=k}^{l-1} d(p_s^i, p^i_{s+1}) = L_i(t^i_l-t^i_k).
\]
In other words, the curves $\gamma_i$ are $L_i$-Lipschitz for $i\in \N$, when restricted to the points $\{t^i_k:k=0,\dots, n(i)\}$. The same Lipschitz bound holds for $\gamma_i$ on $[0,1]$ since the curve is obtained by a linear extension.

First, fix $i \in \N$. We have $\gamma_i(t^i_k) \in X$ for each $i$, and for each $t\in [0,1]$, there is $k=0,\dots, n(i),$ so that 
$$|t-t^i_k|\leq \frac{d(p^i_k,p^i_{k+1})}{L_i}\le\frac{1}{i L_i}.$$ 
Combining this with the Lipschitz bound, we get $\gamma_i$ belongs to the tubular neighborhood $N_{1/i}(X)$ of $X$, of radius $1/i$. 

Now, fix $t\in [0,1]$. We will show that $A_t=\{\gamma_i(t):i\in \N\}$ is precompact in $\ell_\infty(\N)$. To that end, fix  $\eta>0$. Set $N:=\lfloor \frac{1}{\eta} \rfloor+1$. Then, for $i\ge N$, we have $1/i\le \eta$. So, $A_t \subset \{\gamma_1(t), \dots, \gamma_N(t)\}\cup N_\eta(X)$. Since $X$ is totally bounded, it can be covered by finitely many  $\eta$-balls. Therefore, $A_t$ can also be covered by finitely many $\eta$-balls. Since $\eta$ was arbitrary, we have shown that $A_t$ is totally bounded and hence precompact.

Therefore, by the Arzel\`a-Ascoli Lemma \ref{lem:arzelaascoli}, we can pass to a subsequence and assume that $\gamma_i \to \gamma$ converges uniformly to a Lipschitz curve. Further, since $\gamma_i \subset N_{1/i}(X)$, we get that $\gamma \subset X$. Also, $\gamma_i(0) \in E$ and $\gamma_i(1) \in F$, for all $i$'s, so $\gamma(0) \in E$ and $\gamma(1) \in F$.

By \cite[Proposition 4]{keith03}, for each $i\in \N$, we have 
\begin{equation}\label{eq:keithlsc}
\int_\gamma \tilde{g}_i ~ds \leq \liminf_{j\to\infty}\int_{\gamma_j} \tilde{g}_i ~ds.
\end{equation}
By compactness, each function $\tilde{g}_i$ is uniformly continuous on $X$. More precisely, for any $\epsilon>0$, there is a $\delta>0$ so that if $x\in X, y\in \ell_\infty(\N)$ with $d(x,y)\leq \delta$, then $|\tilde{g}_i(x)-\tilde{g}_i(y)|\leq \epsilon$. If $j>\frac{1}{\delta}$, then the curve $\gamma_j$ restricted to $[t^j_k, t^j_{k+1}]$ is a linear segment with end points $p^j_k, p^j_{k+1}$, and thus of length at most $1/j$, and hence at most $\delta$. Now, summing over $k$ together with a Riemann sum upper bound with the partition 
$\{t^j_k\}$, and sending $\epsilon \to 0$ gives
\begin{equation}\label{eq:upperbound1}
\lim_{j \to \infty}\int_{\gamma_j} \tilde{g}_i ~ds -\sum_{k=0}^{n(j)-1}  \tilde{g}_i(p^j_k) d(p^j_k,p^j_{k+1}) = 0.
\end{equation}

Combining Estimates \eqref{eq:upperbound1} and \eqref{eq:keithlsc} with the fact that $\tilde{g}_i$ is increasing in $i$, gives

\begin{eqnarray}\label{eq:upperbound-2}
\int_\gamma \tilde{g}_i ~ds &\stackrel{\eqref{eq:keithlsc}}{\leq}& \liminf_{j\to\infty}\int_{\gamma_j} \tilde{g}_i ~ds \stackrel{\eqref{eq:upperbound1}}{\leq} \liminf_{j\to\infty}\sum_{k=0}^{n(j)-1}  \tilde{g}_i(p^j_k) d(p^j_k,p^j_{k+1})~ds \\
&\stackrel{\tilde{g}_i\leq\tilde{g}_j}{\leq}& \liminf_{j\to\infty}\sum_{k=0}^{n(j)-1}  \tilde{g}_j(p^j_k) d(p^j_k,p^j_{k+1})~ds  \stackrel{\eqref{eq:discchain}}{\leq} 1-\delta. \nonumber
\end{eqnarray}

Sending $i \to \infty$ and using monotone convergence, since $\tilde{g}_i \nearrow g$, we get 
\[
\int_\gamma g ~ds \leq 1-\delta/2.
\]
This is a contradiction to the fact that $g$ is an upper gradient for $u$, because $u(\gamma(1))-u(\gamma(0))=1 \leq \int_\gamma g ~ds$. Indeed, we then must have $\lim_{i\to\infty} a_i = 1$.\\

\noindent\textbf{Conclusion:} We define $u_i = u/a_i$, and $g_i= \tilde{g}_i/a_i$, and the claim follows from the first and second claim.

\end{proof}

Finally, we can prove Theorem \ref{thm:continuous-u}.

\begin{proof}[Proof of Theorem \ref{thm:continuous-u}] Fix $u \in \spn$ as in the claim. Let $|\nabla u|_p$ be the minimal upper gradient of $u$. Let $n\geq 1$. By Vitali-Caratheodory, see Remark \ref{rmk:vitalicarath}, we can find a sequence of upper gradients $\tilde{g}_n\geq \max(|\nabla u|_p,n^{-1})$ of $u$ which are lower semicontinuous and converge in $L^p$ to $|\nabla u|_p$. Let $\tilde{u}_{m,n},\tilde{g}_{m,n}$ be the sequence constructed in Theorem \ref{thm:simplified} with $\tilde{g}_{m,n}$ a local Lipschitz upper gradient for $\tilde{u}_{m,n}$ and $\tilde{g}_{m,n}\to_{L^p} \tilde{g}_n$ as $m\to\infty$. We can choose an index $m(n)$ so that $\|\tilde{g}_{m(n),n}-g_n\|_{L^p}\leq \frac{1}{n}$

Next, let $u_n = \tilde{u}_{m(n),n}$ with a local Lipschitz upper gradient $g_n = \tilde{g}_{m(n),n}$. By construction $g_n\to_{L^p} |\nabla u|_p$ and $u_n|_E=0, u_n|_F=1$. Applying Lemma \ref{lem:local-lipshitz} to $u_n$ and $g_n$, and using the fact that $\partial\{u_n< t \}\subset u_n^{-1}(t)$ for all $t\in [0,1]$, we get
\[
 \int_{0}^1 \int_{\partial\{u_n< t \}} \rho ~d\cH^{N-1} ~dt \leq
\int_{-\infty}^\infty \int_{u_n^{-1}(t)} \rho ~d\cH^{N-1} ~dt \leq \frac{2v_{N-1}}{v_N}\int \rho g_n ~d\cH^{N}.
\]
for any Borel function $\rho \in L^q(X)$. Now, take the limit superior of both sides as $n$ tends to infinity. The right hand-side converges because $\rho\in L^q(X)$.
\end{proof}

\section{Proofs of main theorems}

\begin{proof}[Proof of Theorem \ref{thm:lowerbound}] 
Fix $p,q\in(1,\infty)$ with $p^{-1}+q^{-1}=1$.
If there are no admissible functions for  $\Sigma_H(E,F)$, then $\Mod_q(\Sigma_H(E,F))=\infty$ and the Inequality (\ref{eq:weakduality}) holds trivially. Assume therefore that there is a non-negative Borel function $\rho \in L^q(X)$ which is admissible for $\Sigma_H(E,F)$. As we already noted, $\Mod_p(\Ga(E,F))<\infty.$ So, let $g$ be a non-negative Borel function that is  admissible for $\Mod_p(\Gamma(E,F))$. By Lemma \ref{lem:modcap}, there is a function $u\in N^{1,p}(X)$, with $u|_E=0, u|_F=1$, so that $|\nabla u|_p \leq g$ almost everywhere.
By Theorem \ref{thm:continuous-u} and admissibility of $\rho$ for $\Sigma_H(E,F)$, we obtain a sequence $u_i \in N^{1,p}(X)$ such that
$$1 \leq \limsup_{i 
\to \infty} \int_{0}^1 \int_{\partial\{u_i< t \}} \rho d\cH^{N-1} ~dt \leq \frac{2v_{N-1}}{v_N}\int \rho |\nabla u|_p ~d\cH^N.$$
By H\"older's inequality,
\[
\frac{v_N}{2v_{N-1}} \leq \left(\int \rho^q  ~d\cH^N \right)^{1/q} \left(\int |\nabla u|_p^p  ~d\cH^N \right)^{1/p}\le \left(\int \rho^q  ~d\cH^N \right)^{1/q} \left(\int g^p  ~d\cH^N \right)^{1/p}.
\]
Taking an infimum over all admissible $\rho,g$, yields the inequality. In the case $\Mod_p(\Gamma(E,F))=0$, taking an infimum over $g$ admissible yields a contradiction to $\rho\in L^q(X)$, and thus there are no admissible functions $\rho \in L^q(X)$ for $\Sigma_H(E,F)$. Then, $\Mod_q(\Sigma_H(E,F))=\infty$ as claimed.


\end{proof}

\begin{proof}[Proof of Corollary \ref{cor:plane}] Assume $\sigma\in\Sigma_H(A,C;Q)$. Then, there is a relatively open set $U \subset Q$ with $A \subset U$ and $C$ contained in the relative interior of the complement of $U$, such that $\sigma=\mathcal{H}^1|_{\partial_Q U}$ is finite. By Lemma \ref{lem:contained}, there is a simple rectifiable curve $\gamma\in \Gamma_2(Q)$ contained in $\partial_Q U$ with $\cH^1|_\gamma \leq \cH^1|_{\partial_Q U}$. Hence by Lemma \ref{lem:modulus-contained}, we have $\Mod_q(\Sigma_H(A,C;Q)) \leq \Mod_q(\Gamma_2(Q))$. Here, it is important that $\gamma$ be simple, because when using Definition \eqref{eq:modulusdef}, we need $\int_\gamma \rho ~ds = \int_\gamma \rho ~d\cH^1|_\gamma$ so that the modulus for \emph{curves} $\gamma \in \Gamma_2(Q)$ coincides with the modulus of \emph{measures} $\cH^1|_\gamma$, for $\gamma\in \Gamma_2(Q)$. This is only true for simple curves. The claim follows from Theorem \ref{thm:lowerbound} applied to the metric space $Q$ with its restricted metric and Hausdorff measure.
\end{proof}

\noindent \textbf{Questions:} We leave open a few questions. First,  regarding the case of $p=1$ in Theorem \ref{thm:lowerbound}, we plan to return to this question in later work. It seems, that this case does not lead to significant issues, but that the case $q=\infty$  for the dual modulus needs to be interpreted properly. A second question is when in Theorem \ref{thm:continuous-u} can we replace the sequence $u_i$ by $u$ and obtain a co-area inequality. The issue, formally, is whether the Hausdorff $(N-1)$-measure of $\partial\{u_i<t\}$ converges, for almost every $t$, to the Hausdorff $(N-1)$-measure of the set $\partial\{u<t\}$. However, these measures may fail to be lower semi-continuous. It seems a further assumption may be needed to guarantee sufficient continuity.


\bibliography{pmodulus}{}
\bibliographystyle{plain}

\end{document}